\documentclass{article}
\usepackage{graphicx}

\usepackage[fleqn]{amsmath}

\usepackage{amsthm}
\usepackage{amssymb}
\usepackage{amsbsy}
\usepackage{amsfonts}
\usepackage{mathrsfs}
\usepackage{amsmath} 
\usepackage[all]{xy}
\usepackage{amstext}
\usepackage{amscd}
\usepackage[dvips]{epsfig}
\usepackage{psfrag}
\usepackage{enumerate}
\usepackage{flafter}
\allowdisplaybreaks

\theoremstyle{plain}
\newtheorem{thm}{Theorem}[section]
\newtheorem{prop}[thm]{Proposition}

\newtheorem{lem}[thm]{Lemma}
\theoremstyle{definition}

\newtheorem{rem}[thm]{Remark}
\newtheorem{defn}[thm]{Definition}

\def\Ker{\mathop{\mathrm{Ker}}\nolimits}

\def\Hom{\mathop{\mathrm{Hom}}\nolimits}

\newcommand{\col}{{\rm Col}}

\newcommand{\lra}{\longrightarrow}
\newcommand{\ra}{\rightarrow}

\newcommand{\Z}{{\Bbb Z}}

\newcommand{\pc}[2]{\mbox{$\begin{array}{c}
\includegraphics[scale=#2]{#1.eps}
\end{array}$}}
\begin{document}
\large
\begin{center}
{\bf\Large Twisted cohomology pairings of knots III; triple cup products}
\end{center}
\vskip 1.5pc
\begin{center}{\Large Takefumi Nosaka}\end{center}\vskip 1pc\begin{abstract}\baselineskip=12pt \noindent
Given a representation of a link group,
we introduce a trilinear form, as a topological invariant. 
We show that, if the link is either hyperbolic or a knot with malnormality, then
the trilinear form equals the pairing of the (twisted) triple cup product and the fundamental relative 3-class.
Further, we give some examples of the computation.
\end{abstract}

\begin{center}
\normalsize
\baselineskip=11pt
{\bf Keywords} \\
\ \ \ Cup product, Bilinear form, knot, twisted Alexander polynomial,
group homology, quandle \ \ \
\end{center}

\baselineskip=13pt

\tableofcontents

\large
\baselineskip=16pt
\section{Introduction}
This paper examines topological invariants of trilinear forms, while the previous papers \cite{Nos5} in this series discussed of bilinear forms.
In general, bilinear form arising from Poincar\'{e} duality is a powerful method, 
as in algebraic surgery theory and classification theorems of some manifolds.
In contrast, there are not so many studies of trilinear forms.
However, some 3-forms and trilinear cup products appear in 3-dimensional geometry
together with topological information (see, e.g., \cite{CGO,Mark,L,MS,Tur}).
For example, we mention interesting observations from the Chern-Simons invariant (or $\phi^3$-theory) of the form
$$ \frac{k }{2\pi } \int_{M} \mathrm{tr}( A \wedge d A +\frac{2}{3}A \wedge A \wedge A) . $$

We give the definition of trilinear pairing (see \eqref{kiso}) in a general situation where the coefficients are arbitrary.
Let $Y$ be a compact 3-manifold with toroidal boundary with orientation 3-class
$[Y, \partial Y] \in H_3 ( Y ,\partial Y ;\Z ) \cong \Z.$
Choose a group homomorphism $\pi_1(Y) \ra G $,
a right $G$-module $M$, and a $G$-invariant trilinear function $ \psi : M^3 \ra A $ over a ring $A$.
Then, we can define the composite map
\begin{equation}\label{kiso}
H^1( Y,\partial Y ; M )^{ \otimes 3} \xrightarrow{\ \ \smile \ \ } H^{3}( Y ,\partial Y ; M^{\otimes 3} )
\xrightarrow{ \ \ \langle \bullet , \ [Y, \partial Y] \rangle \ \ } M^{\otimes 3 } \xrightarrow{\ \  \psi, \ \ }A
. \end{equation}
Here $M$ is regarded as the local coefficient of $Y$ via $f$, and
the first map $\smile$ is the cup product, and the second (resp. third) is defined by the pairing with $[Y, \partial Y]$ (resp. $\psi$).
In contrast to this definition, the 3-form \eqref{kiso} is considered to be something uncomputable. Actually, it seems hard to concretely deal with the 3-class $[Y, \partial Y] $ and the cup products.

This paper addresses the link case
where $Y$ is the 3-manifold which is obtained from the 3-sphere by removing an open tubular neighborhood of a link $L$, i.e., $Y=S^3 \setminus \nu L.$
In fact, if $L$ is a hyperbolic link, we obtain a diagrammatic method of computing the trilinear pairings.
To be precise, in Section \ref{sss3}, starting from a link diagram, we define invariants of trilinear forms, and show (Theorem \ref{mainthm3}) that the invariant is equal to \eqref{kiso}, if $L$ is a hyperbolic link. In addition, we also show a similar theorem in the torus case (see Theorem \ref{mainthm4}).
The point in the theorem is that, in the computation, we do not need describing $[Y, \partial Y] $ and cup products; thus, this computation is not so hard. In fact, we give some examples; see \S \ref{Lb162r3}.
In addition, as an application (Theorem \ref{4i1}), when $Y$ is a 3-fold covering space of $S^3$ branched along a hyperbolic link $L$ and $M$ is a trivial coefficient, we give a diagrammatic computation of the trilinear pairing \eqref{kiso}.

This paper is organized as follows. Section 2 formulates the trilinear forms in terms of the quandle cocycle invariants,
and states the main theorems. Section 3 discusses a relation to 3-fold branched coverings. Section 4 describes some computations.
Section 5 gives the proofs of the theorems.

\

\noindent
{\bf Notation.} Every link $L$ is smoothly embedded in the 3-sphere $S^3$ with orientation.
We write $E_L$ for the 3-manifold which is obtained from $S^3$ by removing an open neighborhood of $L$.

\section{Results; diagrammatic formulations of the trilinear forms}\label{ss1}
Our purpose in this section is to give a link invariant of trilinear form (Theorem \ref{alspw23}), and to state the main results in \S \ref{sss3}.
For this purpose, \S \ref{sss2} starts by reviewing colorings,
and formulates some link-invariants of linear forms.

Thorough this section, we fix a group $G$ and a right $G$-module $M$ over a ring $A$.

\subsection{Preliminary; the formulations of the first cohomology}\label{sss2}

We need some notation from \cite{IIJO,Nos5} before proceeding.
Denote $ M \times G$ by $X$.
Further, define a binary operation on $X$ by
\begin{equation}\label{kihon} \lhd: (M \times G) \times (M \times G)\lra M \times G, \ \ \ \ \ \ (a,g,b,h) \longmapsto (\ (a-b)\cdot h +b, \ h^{-1}gh \ ), \end{equation}
which was first introduced in \cite[Lemma 2.2]{IIJO}, and satisfies ``the quandle axiom".
Furthermore, we choose a link $L \subset S^3$ with a group homomorphism $f:\pi_1(S^3\setminus L) \ra G$.

Next, we review colorings.
Choose an oriented diagram $D$ of $L.$
Then, it follows from 
the Wirtinger presentation of $D$ that the homomorphism $f$ is regarded as a map $ \{ \mbox{arcs of $D$} \} \to G$.
Furthermore, a map $\mathcal{C}: \{ \mbox{arcs of $D$} \} \to X$ is an $X$-{\it coloring} if
it satisfies 
$\mathcal{C}(\alpha_{\tau}) \lhd \mathcal{C}(\beta_{\tau}) = \mathcal{C}(\gamma_{\tau})$ at each crossings of $D$ illustrated as
Figure \ref{fig.color}.
It is worth noticing that the set of all colorings is regarded as a subset of the direct product $X^{\alpha_D }$, where $\alpha_D$ is the number of arcs of $D$.
Let $\mathrm{Col}_X(D_f) $ denote the set of all $X$-colorings over $f$, that is,
\begin{equation}\label{kihon2294} \mathrm{Col}_X(D_{f}):= \{ \ \mathcal{C} \in (M\times G)^{\alpha_D } \ | \ \mathcal{C} \ \textrm{is an }X\textrm{-coloring}, \ \ p_G \circ \mathcal{C} =f \ \}, \end{equation}
where $p_G$ is the projection $ X = M \times G \ra G$.
Then, we can easily verify from the linear operation \eqref{kihon} that $\mathrm{Col}_X(D_{f})$ is made into an abelian subgroup of $M^{\alpha (D)}$, 
and that the diagonal subset $M_{\rm diag} \subset M^{\alpha _D } $ is 
a direct summand in $\mathrm{Col}_X(D_{f}) $.
Denoting another summand by $ \mathrm{Col}^{\rm red}_X(D_{f}) $, we have a decomposition
$ \mathrm{Col}_X(D_{f}) \cong \mathrm{Col}^{\rm red}_X(D_{f}) \oplus M_{\rm diag} . $

The previous paper \cite{Nos5} gave a topological meaning of the coloring sets as follows:
\begin{thm}[\cite{Nos5}]\label{mai1}Let $E_L$ be a link complement in $S^3$ as in \S 1.
Regard the $G$-module $M$ as a local system of $E_L$ via 
$f: \pi_1(E_L) \ra G$. Then, there are isomorphisms
\begin{equation}\label{g21gg33} \mathrm{Col}_{X } (D_{f}) \cong H^1(E_L , \ \partial E_L ;M ) \oplus M, \ \ \ \ \ \ \ \ \ \ \ \mathrm{Col}_{X }^{\rm red} (D_{f}) \cong H^1(E_L , \ \partial E_L ;M ) . \end{equation}
\end{thm}
Furthermore, let us review shadow colorings \cite{CKS,IIJO}.
A {\it shadow coloring} is a pair of a coloring $\mathcal{C} $ over $f$ and a map $\lambda $ from the complementary regions of $D$ to $M$,
satisfying the condition depicted in the right side of Figure \ref{fig.color} for every arcs.
Let $\mathrm{SCol}_{X}(D_f )$ denote the set of shadow colorings of $D$ such that
the unbounded exterior region is assigned by $0 \in M$.
Notice that, by the coloring rules, assignments of the other regions are uniquely determined from the unbounded region, and admit, therefore, a shadow coloring; 
we thus obtain a bijection
\begin{equation}\label{kihos12} \col_X (D_f) \simeq \mathrm{SCol}_{X } (D_f). 
\end{equation}\vskip -1.419937pc
\begin{figure}[htpb]
\begin{center}
\begin{picture}(100,70)
\put(-142,46){\Large $\alpha_{\tau} $}
\put(-86,45){\Large $\beta_{\tau} $}
\put(-86,12){\Large $\gamma_{\tau} $}
\put(-132,27){\pc{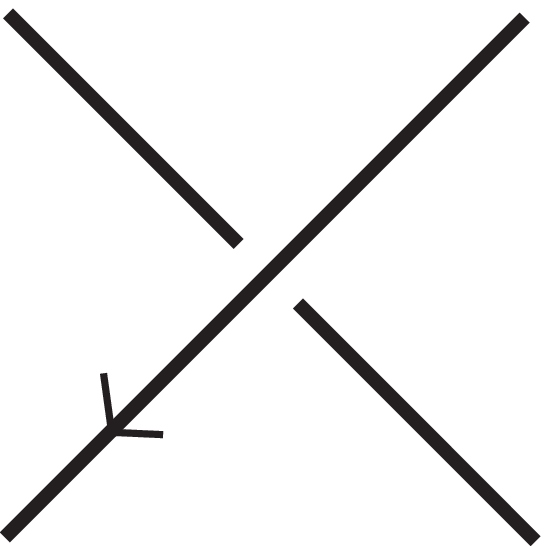}{0.2530174}}

\put(51,32){\pc{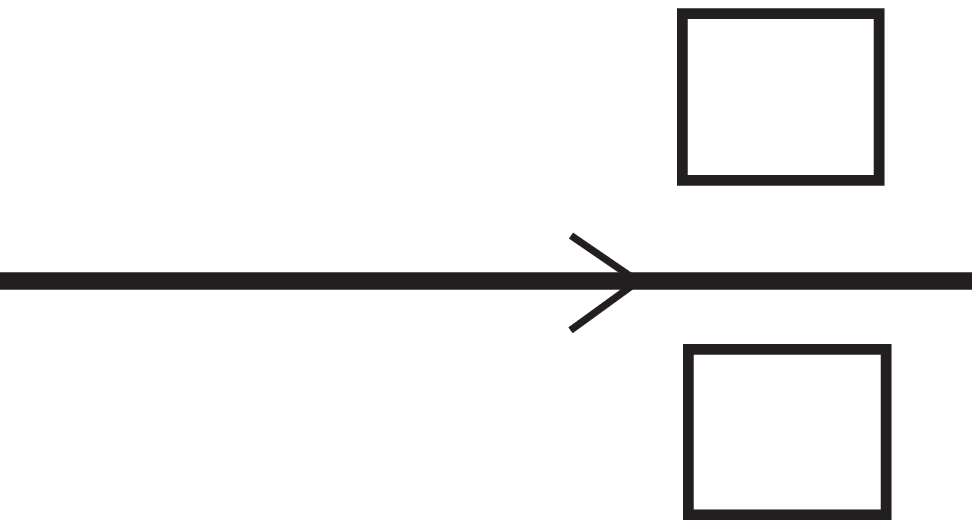}{0.286}}
\put(-65,16){$\mathcal{C}(\alpha_{\tau}) \lhd \mathcal{C}(\beta_{\tau}) = \mathcal{C}(\gamma_{\tau})$}

\put(78,43){\large $ \delta$}
\put(116,19){\large $R $}
\put(114,46){\large $ R ' $}
\put(154,35){\large $ \lambda(R')= (\lambda(R)-b) \cdot h +b $, }
\put(154,21){\large where $ \mathcal{C} (\delta )= (b,h) .$}
\end{picture}
\end{center}
\vskip -1.9937pc
\caption{The coloring conditions at each crossing $\tau$ and around each arcs. 
\label{fig.color}}
\end{figure}

\subsection{Invariants of trilinear forms}\label{sss3}

In addition, we will explain Definition \ref{deals23} below, and show Theorem \ref{alspw23}.

For this, we need two things:
first, we take three $G $-modules $M_1 ,\ M_2,\ M_3 $ and the associated $X_i = M_i \times G. $ Let $A$ be an abelian group.
On the other hand, we prepare a trilinear map
$\psi : M_1 \times M_2 \times M_3 \ra A$ over $\Z$ satisfying
the $G$-invariance, that is,
\begin{equation}\label{skew4}\psi(a_1 \cdot g ,a_2 \cdot g ,a_3 \cdot g ) =\psi(a_1,a_2,a_3 ) ,\end{equation}
holds for any $a_i \in M_i$ and $g \in G$.

Next, let us consider the map $ X_1 \times X_2 \times X_3 \ra A $ by the formula
\begin{equation}\label{skew} \bigl( (b_1,g_1),(b_2,g_2),( b_3,g_3) \bigr)\longmapsto \psi \bigl( ( b_1-b_2) \cdot(1-g_2 ),\ b_2 -b_3 , \ b_3 - b_3 \cdot {g_3}^{-1} \bigr), \end{equation}
for $a_i \in M_i$ and $g_1,g_2,g_3 \in G $. This map was first defined in \cite[Corollary 4.6]{Nos2}.
Furthermore, given three shadow colorings $ \mathcal{S}_i \in \mathrm{SCol}_{X_i }(D_{f})$ with $i \leq 3$ and each crossing $\tau$ of $D$, we can find assignments as illustrated in Figure \ref{kout1144}.
Inspired by the formula \eqref{skew}, we define a weight of $\tau$ to be
$$ \mathcal{W}_{\psi, \tau }( \mathcal{S}_1, \mathcal{S}_2, \mathcal{S}_3 ):= \psi \bigl( (a_1-b_1)(1-g^{ \epsilon_{\tau}}),b_2-c_2 , c_3 -c_3 \cdot h^{-1} \bigr) \in A,$$
where $ \epsilon_{\tau} \in \{ \pm 1 \}$ is the sign of $\tau.$

\

\begin{figure}[htpb]
\ \ \ \ \ \ \ \ \ \ \ \ \ \ \ \ \ \ \ \ \ \ \ \ \ \ \ \ \ \ \ \ \begin{picture}(100,26)
\put(-89,5){\large \fbox{$a_1$}}
\put(59,5){\large \fbox{$a_2$}}
\put(199,5){\large \fbox{$a_3$}}
\put(-99,25){\large $( b_1,g)$}
\put(-22,25){\large $( c_1,h)$}
\put(51,25){\large $( b_2,g)$}
\put(131,25){\large $( c_2,h) $}
\put(191,25){\large $( b_3,g)$}
\put(271,25){\large $( c_3,h) \in M \times G $}

\put(-76,3){\pc{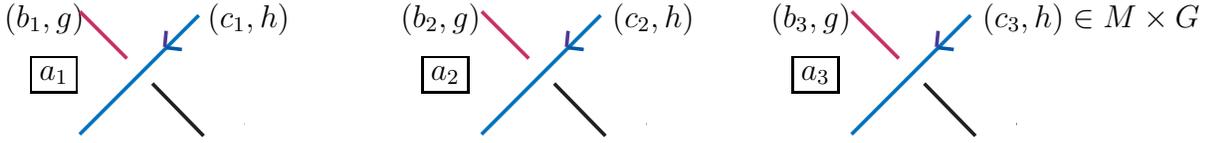}{0.2722792447304}}
\put(76,3){\pc{kouten2}{0.272266682447304}}
\put(216,3){\pc{kouten2}{0.272266682447304}}

\end{picture}
\vskip 1.1785pc

\caption{\label{kout1144} Colors around a crossing with respect to three shadow colorings. }
\end{figure}

\begin{defn}\label{deals23}
Given a $G$-invariant trilinear map $\psi : M_1 \times M_2 \times M_3 \ra A$,
we define a trilinear map
$$ \mathcal{T}_{\psi } : \prod_{i=1}^3 \mathrm{SCol}_{X_i }(D_{f}) \lra A; \ \ \ \ \ \ \ \ \
( \mathcal{S}_1, \mathcal{S}_2, \mathcal{S}_3 ) \longmapsto \sum_{\tau} \mathcal{W}_{\psi, \tau }( \mathcal{S}_1, \mathcal{S}_2, \mathcal{S}_3 ), $$
where $\tau $ runs over all the crossings of $D$.
\end{defn}

The point is that, given a diagram $D$,
we can diagrammatically deal with the trilinear $ \mathcal{T}_{\psi } $ by definitions;
see \S \ref{Lexa1}--\S \ref{secex3} for examples. 

Next, we now show the invariance of $\mathcal{T}_{\psi }$ up to trilinear equivalence:
\begin{thm}\label{alspw23}
Let two diagrams $D$ and $D'$ differ by a Reidemeister move.
There is a canonical isomorphism 
$\mathcal{B}_{i}: \mathrm{SCol}_{X_i}(D_{f}) \simeq \mathrm{SCol}_{X_i}(D_f') $,
for which the equality $ \mathcal{T}_{\psi }= \mathcal{T}_{\psi }' \circ (\mathcal{B}_1 \otimes \mathcal{B}_2 \otimes \mathcal{B}_3) $ holds as a map.

In particular, the equivalence class of the trilinear map $ \mathcal{T}_{\psi }$
depends on only the homomorphism $f:\pi_1(S^3 \setminus L) \ra G $ and the input data $( M_1,M_2, M_3, \psi)$.
\end{thm}
\begin{proof}
We first focus on Reidemeister move of type III;
see Figure \ref{kout11}.
Then, considering the correspondence in Figure \ref{kout11} with $x_i,y_i,z_i \in X_i$, we have the bijection $\mathcal{B}_{i}$.
Moreover, we suppose that the left region is colored by $r_i \in M$.
Thus, it is enough to show the desired equality.
For this, take $a_i,b_i,c_i \in M_i$ and $g,h,k \in G$ such that $x_i = (a_i,g), \ y_i = (b_i,h), \ z_i= (c_i,k) \in X_i$.
Then, the sum from the left side is, by definition and examining the figure, computed as
\[ \psi\bigl((r_1 - a_1) (1-g), \ a_2 - c_2, \ c_3 (1-k^{-1})\bigr) + \psi \bigl((r_1 g- a_1 g+ a_1 -b_1)(1-h), \ b_2 - c_2, \ c_3 (1-k^{-1})\bigr) \]
\[ + \psi \bigl((r_1 - a_1) k (1-k^{-1}gk), \ (a_2 - b_2)k, \ (b_3 k -c_3k +c_3) (1-k^{-1}h^{-1}k)\bigr) . \]
On the other hand, the sum from the right side is formulated as
\[ \psi\bigl((r_1 - a_1) (1-g), \ a_2 - b_2, \ b_3 (1-h^{-1})\bigr) + \psi\bigl((r_1 - b_1) (1-h), \ b_2 - c_2, \ c_3 (1-k^{-1})\bigr) \]
\[ + \psi \bigl((r_1 - a_1) h(1-h^{-1}gh), \ (b_2 - a_2)h+a_2 -c_2, \ c_3 (1-k^{-1})\bigr) . \]
Then, an elementary calculation from \eqref{skew4} can show the two sums are equal.
However, since the calculation is a little tedious, we omit the detail.

Finally, the required equality concerning Reidemeister moves of type I immediately follows from $\psi (0,y,z)=0$, and
the invariance of type II is clear by a similar discussion.
\end{proof}

\

\

\vskip 0.45pc
\vskip 0.45pc
\vskip 0.45pc
\begin{figure}[htpb]
\begin{picture}(20,0)

\put(69,-2){\Large \fbox{$r_i$}}
\put(269,-2){\Large \fbox{$r_i$}}

\put(89,48){\large $x_i$ }
\put(123,49){\large $y_i$ }
\put(161,48){\large $z_i$ }

\put(290,46){\large $x_i$ }
\put(328,47){\large $y_i$ }
\put(366,46){\large $z_i$ }

\put(88,-54){\large $z_i$ }
\put(114,-54){\large $y_i \lhd z_i $ }
\put(153,-54){\large $(x_i\lhd z_i) \lhd ( y_i \lhd z_i)$ }

\put(290,-54){\large $z_i$ }
\put(318,-54){\large $y_i \lhd z_i $ }
\put(356,-54){\large $ (x_i\lhd y_i) \lhd z_i $ }
\put(211,-4){\huge $ \longleftrightarrow $ }
\put(26,-9){\pc{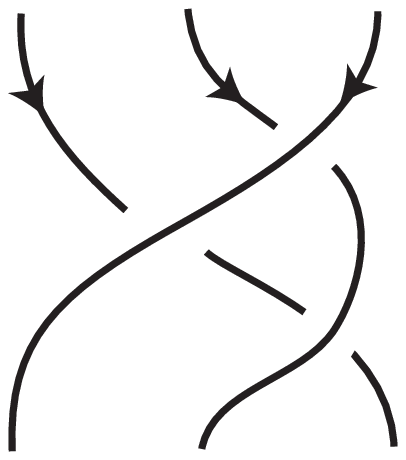}{0.679627304}}
\put(226,-9){\pc{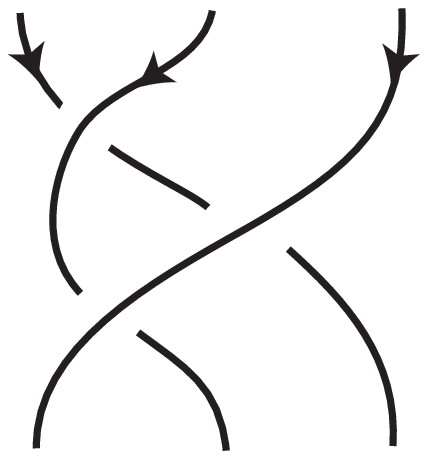}{0.679627304}}
\end{picture}

\vskip 4.285pc

\caption{\label{kout11} The 1:1-correspondence associated with a Reidemeister move of type III. }
\end{figure}

\begin{rem}\label{als1p32w111}
In this way, the construction for trilinear forms is applicable to not only tame links in $S^3$,
but also handlebody-knots $\mathcal{H}_g$ in $S^3$.
In fact, as a similar discussion to 
\cite{IIJO}, we can easily check that the trilinear form is invariant with respect to the diagrammatic moves of handlebody-knots; see \cite[Figures 1 and 2]{IIJO} for the moves.
\end{rem}

\subsection{Topological meaning of the trilinear forms}\label{sss4}


As mentioned in the introduction,
we will show (Theorems \ref{mainthm3} and \ref{mainthm4}) that the trilinear forms of some links are equal to
the trilinear pairings (The proofs of the theorems appear in \S \ref{Lbo33222}). 

\begin{thm}\label{mainthm3}
Let $M_1, M_2, M_3$ be $G$-modules as in Definition \ref{deals23}. 
Furthermore, choose a fundamental class $ [ E_L ,\partial E_L ] $ in $ H_3( E_L ,\partial E_L ;\Z )\cong \Z $.

We assume that $L$ is either a hyperbolic link or a prime knot which is neither a cable knot nor a torus knot.
Then, via the identification \eqref{g21gg33},
the trilinear form $ \mathcal{T}_{\psi}$ 
is equal to the following composite map:
\begin{equation}\label{skew8} \bigotimes_{i: \ 1\leq i \leq 3 } H^1( E_L ,\partial E_L ; M_i ) \xrightarrow{\ \ \smile \ \ } H^3( E_L ,\partial E_L ; M_1 \otimes M_2 \otimes M_3 ) \xrightarrow{ \  \psi\circ \langle \bullet , [ E_L ,\partial E_L ] \rangle \
} A . \end{equation}
\noindent
\end{thm}

In addition, we mention the torus knot, although we need a condition.
More precisely,
\begin{thm}\label{mainthm4}
Let $M_1, M_2, M_3$, $\psi$, and $ [ E_L ,\partial E_L ] $ be as above. 

Assume that $L$ is the $(m,n)$-torus knot.
Then, the trilinear form $\mathcal{T}_{\psi}$ is equal to the composite \eqref{skew8} modulo the integer $nm \in \Z $.
\end{thm}

As a concluding remark,
while the triple cup product of a link often is considered to be speculative and uncomputable, it become 
computable from only a link diagram without describing $[E_L, \partial E_L]$ and any triangulation in $S^3 \setminus L$.


\begin{rem}\label{cl31l}
Finally, we compare the trilinear forms in Definition \ref{deals23} with the existing results on ``the quandle cocycle invariants", in detail.
Briefly speaking, the link invariant in \cite{CKS} is constructed from
a quandle $X$ and a map $\Phi: X^3 \ra A $
which satisfy ``the quandle cocycle condition",
and is defined to be a certain map $ \mathcal{J}_{\Phi} \ : \mathrm{SCol}_X(D )\ra A$.
Then, we note that
our trilinear form is a trilinearization of the quandle cocycle invariants with respect to quandles of the form $X= M \times G $.
To be precise, 
if $M = M_1=M_2=M_3$, we can see that
the associated invariant
$\mathcal{J}_{\Phi} : \mathrm{SCol}_X(D )\ra A$
is equal to the composite
$\mathcal{T}_{\psi} \circ (\bigtriangleup \times {\rm id}) \circ \bigtriangleup$ by definitions.
In conclusion, the theorems also suggest topological meanings of the quandle cocycle invariants with $X= M \times G $.

\end{rem}

\section{Relation to 3-fold branched coverings}\label{S4}
Although we considered relative cohomology, for $n\in \mathbb{Z}_{\geq 0}$ and a closed 3-manifold $N$, let us consider the triple cup product
\begin{equation}\label{kiso8} H^1(N;\Z/n\Z)^{\otimes 3}\xrightarrow{\ \ \smile \ \ } H^3(N;\Z/n\Z)
\xrightarrow{ \ \ \langle \bullet , \ [N, \partial N] \rangle } \Z /n ,\end{equation}
where the coefficient module $\Z/n$ is trivial. Although there are studies of this map (see, e.g., \cite{MS,CGO,Tur}), there are few examples of the computation.
As an application of the theorems above, this section gives a recover of the triple cup products of $N$, when $N$ is a 3-fold cyclic covering of $S^3 $ branched over a link.

To state Theorem \ref{4i1}, we need some terminology.
Let $G$ be $\Z/3= \langle t | t^3=1 \rangle $.
Consider the epimorphism $f: \pi_1(S^3 \setminus L ) \ra G$ which sends every meridian to $t$, and the associated 3-fold cyclic branched covering $\widetilde{C}_L \ra S^3 $.

\begin{thm}\label{4i1}
Let $M_1$, $M_2$, and $M_3$ be $\Z[t^{\pm 1}]/ (n, t^2+t+1)$. 
Let $p:\Z[t^{\pm 1}]/ (n, t^2 + t+1) \ra \Z/n$ be the map which sends $ a+tb$ to $a$.
Set up the map $\psi_0 : M^3 \ra \Z/n$ which takes $ (x,y,z)$ to $ xyz$.
As in Theorem \ref{mainthm3}, assume that $L$ is either a hyperbolic link or a prime knot which is neither a cable knot nor a torus knot. 

Then, there is an isomorphism $ \mathrm{Col}_{X_i}^{\rm red}(D_f) \cong H^1(\widetilde{C}_L;\Z/n\Z) $ such that the trilinear map $\mathcal{T}_{p \circ \psi_0}$ is equivalent to \eqref{kiso8} with $N= \widetilde{C}_L $.
\end{thm}

\begin{proof}[Proof of Theorem \ref{4i1}] We first show the isomorphism $ \mathrm{Col}_{X_i}^{\rm red}(D_f) \cong H^1(\widetilde{C}_L;\Z/n\Z) $.
Let $R$ be the ring $ \Z[t]/(n, t^2 + t+1)$. 
By Theorem \ref{mai1}, we have
$\mathrm{Col}_{X_i}^{\rm red}(D_f) \cong H^1(E_L ;\partial E_L ; M) $.
Notice that $H^i(\partial E_L ;M) $ is annihilated by $1-t$. Since
$1-t$ and $1 + t+t^2$ are coprime, we have
\begin{equation}\label{kiso9} H^1(E_L ;\partial E_L ; M) \cong H^1(E_L ; M)\cong \Hom_{ R\textrm{-mod}}( H_1(E_L ,M), R) .\end{equation}
Let $\widetilde{E}_L \ra E_L= S^3 \setminus L$ be the 3-fold covering.
Then, by Shapiro's Lemma (see, e.g. \cite{Bro}), the canonical inclusion $\iota: \Z/n \ra R$ yields the isomorphisms:
\begin{equation}\label{kiso11} H^*( \widetilde{E}_L: \Z/n ) \cong H^*( E_L; \Z[t]/(n, t^3-1)) \cong H^*( E_L; R) \oplus H^*( E_L; \Z[t]/(n,t-1)) .\end{equation}
Here, the second isomorphism is obtained from the ring isomorphism $\Z[t]/(n, t^3-1) \cong R \oplus \Z[t]/(n, t-1) $.
Let $ i : \widetilde{E}_L \hookrightarrow \widetilde{C}_L$ be the inclusion.
According to \cite[Theorem 5.5.1]{Kaw}, the homology $H_1( \widetilde{C}_L; \Z) $ is annihilated by $1+t+t^2$, and the induced map $i_*: H_1( \widetilde{E}_L; \Z ) \ra H_1( \widetilde{C}_L; \Z )$ is a splitting surjection.
Thus, dually, the induced map $i^*: H^1( \widetilde{C}_L; \Z/n ) \ra H^1( \widetilde{E}_L;\Z /n)$ is injective and the image is isomorphic to $H^1( E_ L; R)$.
In summary, we obtained the desired isomorphism.

We will complete the proof.
By \eqref{kiso11}, we have a splitting injection $\mathcal{S}: H^*( \widetilde{E}_L  ,\partial \widetilde{E}_L ;R ) \ra H^*( E_L ,\partial E_L;\Z/n ) $.
Take the canonical maps $ j: ( \widetilde{E}_L ,\partial \widetilde{E}_L) \ra ( \widetilde{C}_L, \widetilde{C}_L \setminus \widetilde{E}_L ) $ and $ k :( \widetilde{C}_L, \emptyset) \ra ( \widetilde{C}_L, \widetilde{C}_L \setminus \widetilde{E}_L) $.
Then, we have the commutative diagrams on the cup products:
$${\normalsize
\xymatrix{
H^1( E_L, \partial E_L ; M)^{\otimes 3}\ar[d]^{\mathcal{S} } \ar[r]^{\!\!\!\!\! \psi_0 \circ \smile} & H^3( E_L, \partial E_L ; M ) \ar[d]^{\mathcal{S} }\ar[rrr]^{\ \ \ \ \ \ \ \ \ \ \ \ \langle \bullet, [E_L,\partial E_L] \rangle } & & & R \\
H^1( \widetilde{E}_L ,\partial \widetilde{E}_L ; \Z/n)^{\otimes 3} \ar[r]^{\!\!\!\!\!\!\!\! \smile} & H^3 ( \widetilde{E}_L ,\partial \widetilde{E}_L ; \Z/n ) \ar[rrr]^{\ \ \ \ \ \ \ \ \ \ \ \ \langle \bullet, [\widetilde{E}_L ,\partial \widetilde{E}_L ] \rangle }& & & \Z/n \ar@{=}[d] \ar[u]^{\iota }\\
H^1( \widetilde{C}_L, \widetilde{C}_L \setminus \widetilde{E}_L ; \Z/n)^{\otimes 3} \ar[r]^{\!\!\!\!\!\!\!\! \smile} \ar[u]^{\cong }_{j^*}\ar[d]^{k^*} & H^3 ( \widetilde{C}_L, \widetilde{C}_L \setminus \widetilde{E}_L; \Z/n) \ar[u]^{\cong }_{j^*}\ar[d]^{k^*} \ar[rrr]^{\ \ \ \ \ \ \ \ \ \ \ \ \langle \bullet, [ \widetilde{C}_L, \widetilde{C}_L \setminus \widetilde{E}_L] \rangle }& & & \Z/n \ar@{=}[d]
\\
H^1( \widetilde{C}_L ; \Z/n)^{\otimes 3} \ar[r]^{\!\!\!\!\!\!\!\! \smile} & H^3 ( \widetilde{C}_L; \Z/n) \ar[rrr]^{\ \ \ \ \ \ \ \ \ \ \ \ \langle \bullet, [ \widetilde{C}_L] \rangle }& & & \Z/n .
}}
$$
Here, the vertical maps $j^*$ are isomorphisms by the excision axiom.
Moreover, by the discussion in the above paragraph, the composite $k^* \circ (j^*)^{-1}\circ \mathcal{S}$ is an isomorphism from $ H^1( E_L, \partial E_L ; M) $.
Hence, since $ p \circ \iota :\Z/n \ra \Z/n $ is an isomorphism, the following two composites are equivalent:
$$ p\circ \psi_0 \circ \langle \bullet, [E_L,\partial E_L] \rangle \circ \smile, \ \ \ \ \ \langle \bullet, [\widetilde{C}_L] \rangle \circ \smile.$$
By Theorem \ref{mainthm3}, the left hand side is equal to the trilinear map $\mathcal{T}_{p\circ \psi_0}$.
Hence, $\mathcal{T}_{p \circ \psi_0}$ is equivalent to \eqref{kiso8} with $N= \widetilde{C}_L $ as desired.
\end{proof}

\section{Examples as diagrammatic computations}\label{Lb162r3}

\subsection{General situation for the trefoil knot and the figure eight knot}\label{Lexa1}

\begin{figure}[htpb]
\begin{center}
\begin{picture}(10,60)
\put(-163,23){\large $\alpha $}
\put(-148,45){\large $\beta $}
\put(-91,23){\large $\gamma $}
\put(-23,51){\large $\alpha_2 $}
\put(-44,19){\large $\alpha_1 $}
\put(15,10){\large $\alpha_4 $}
\put(37,19){\large $\alpha_3 $}

\put(88,51){\large $\alpha_1 $}
\put(87,-2){\large $\alpha_m $}
\put(88,24){\large $\alpha_i $}

\put(91,10){\normalsize $\vdots $}
\put(91,38){\normalsize $\vdots $}

\put(-161,22){\pc{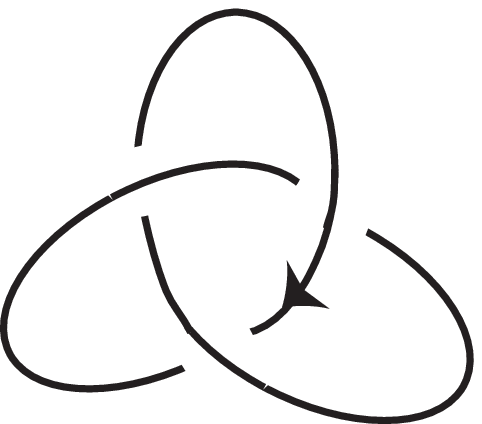}{0.53104}}
\put(-36,22){\pc{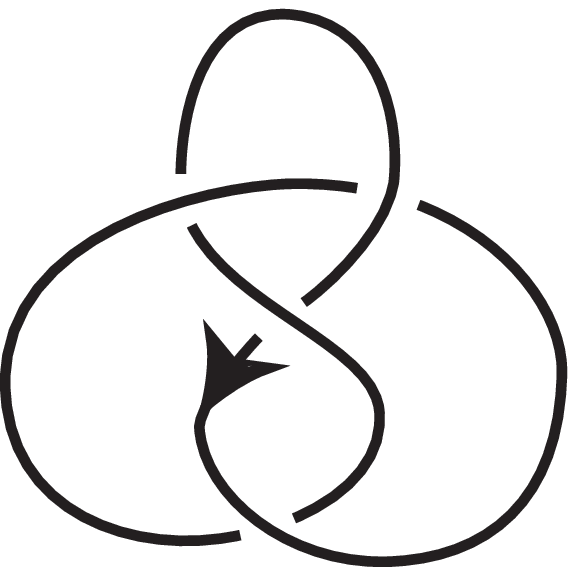}{0.402104}}
\put(96,26){\pc{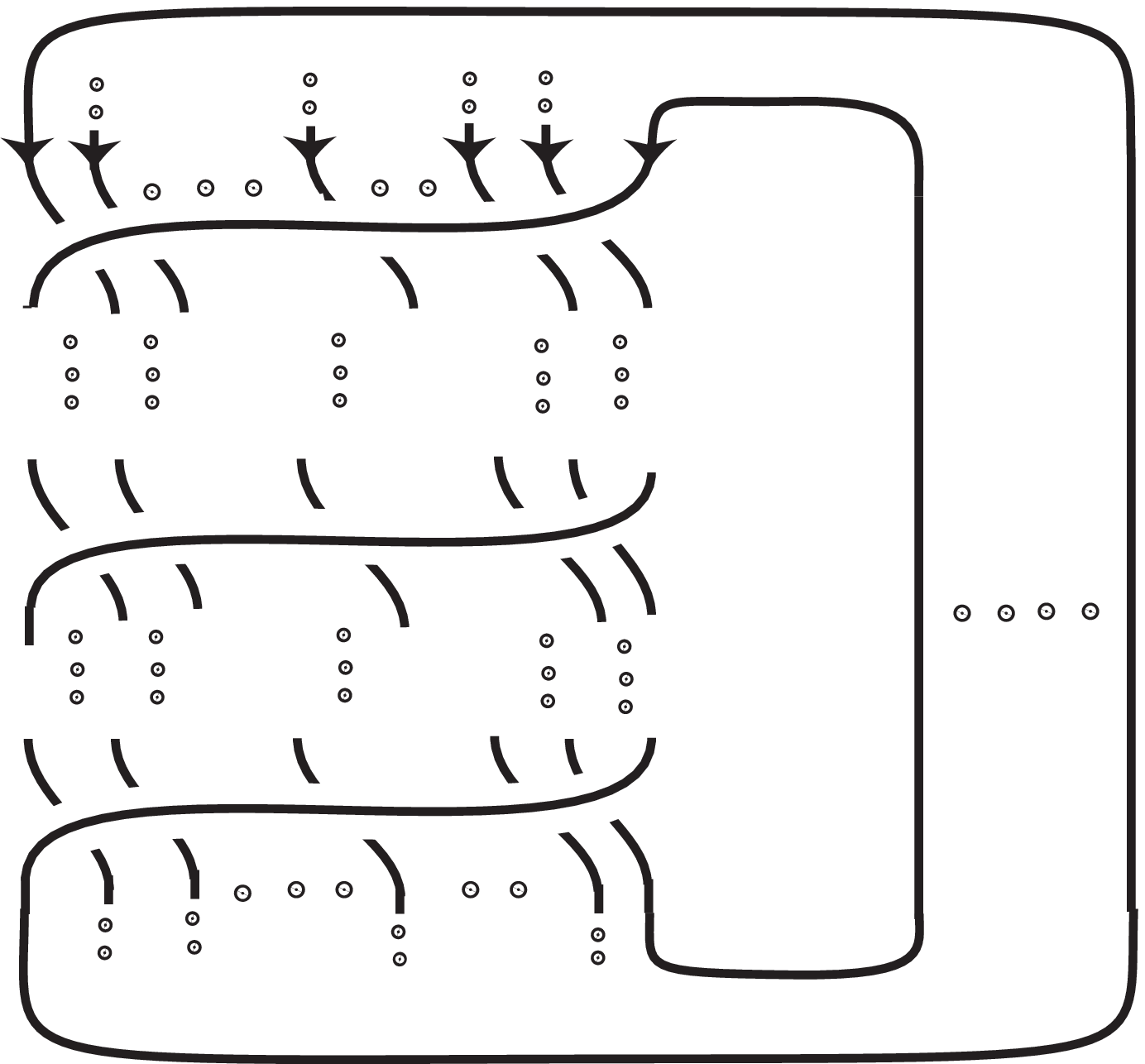}{0.2473104}}
\end{picture}
\end{center}\caption{\label{ftf} The trefoil knot and the figure eight knot}
\end{figure}

We will compute the trilinear forms $ \mathcal{T}_{\psi}$
associated with some homomorphisms $f : \pi_L \ra G$, where $L$ is either the trefoil knot or the figure eight knot. 

As a simple example,
we will focus on the the trefoil knot $3_1 $.
Let $D$ be the diagram of $K$ as illustrated in Figure \ref{ftf}.
Note the Wirtinger presentation $\pi_L \cong \langle \alpha, \beta \ | \ \alpha \beta \alpha =\beta \alpha \beta \rangle . $
Then, we can easily see
that a correspondence $ \mathcal{C}: \{ \alpha, \beta , \gamma \} \ra X$
with
$$\mathcal{C}(\alpha )=(a_i ,g) , \ \ \ \ \mathcal{C}(\beta )=(b_i ,g), \ \ \ \ \mathcal{C}(\alpha )=(c_i ,g) \ \in M_i \times G $$
is an $X$-coloring $ \mathcal{C}$ over $f: \pi_L \ra G$,
if and only if it satisfies the four equations
\begin{equation}\label{eweq22} g = f(\alpha ), \ \ h = f(\beta ), \ \ \ ghg =hgh , 
\end{equation}
\begin{equation}\label{eq22} c_i = a_i \cdot h +b_i \cdot (1-h ), \end{equation}
\begin{equation}\label{eq221} (a_i -b_i ) \cdot (1-g +hg)=(a_i -b_i ) \cdot (1-h +gh)=0 . \end{equation}
Furthermore, given a $G$-invariant linear form $\psi$,
the sum $\mathcal{T}_{\psi}$ is equal to
\noindent
\[ \psi \bigl( -a_1 \cdot (1-g), \ a_2-b_2, \ a_3\cdot (1-h^{-1})\bigr) + \psi \bigl( -b_1 \cdot (1-h), \ b_2-c_2, \ c_3 \cdot (1-h^{-1}g^{-1} h ) \bigr)\]
\[ \ \ + \psi \bigl( -c_1 \cdot (1-h^{-1}g h ), \ c_2-a_2, \ a_3 \cdot (1-g^{-1})\bigr), \]
by definition. Then, by canceling out $c_i$ by using \eqref{eq22} and \eqref{eweq22},
we 
can easily obtain the resulting computation: for $((a_i,g),( b_i,h)) \in {\rm SCol}_{X_i} (D_{f})\subset M_i^2 $,
\begin{equation}\label{eq22131} \mathcal{T}_{\psi} \bigl( (a_1, b_1),( a_2, b_2),( a_3, b_3 ) \bigr)= \psi \bigl(
(a_1 -b_1)g^{-1} , \ (a_2 -b_2)\cdot h , \ a_3 -b_3 \bigr) \in A. \end{equation}

\

Next, we will compute $\mathcal{T}_{\psi} $ of the figure eight knot.
However, the computation can be done in a similar way to the trefoil case. So we only describe the outline.

Let $D$ be the diagram with arcs as illustrated in Figure \ref{ftf}.
Similarly, we can see
that a correspondence $ \mathcal{C}: \{ \alpha_1, \alpha_2, \alpha_3, \alpha_4\} \ra X$
with $\mathcal{C}(\alpha_i )=(x_i ,z_i) \in M_i \times G $
is an $X$-coloring $ \mathcal{C}$ over $f: \pi_L \ra G$,
if and only if it satisfies the following equations:
\begin{equation}\label{eq225dd} z_i = f(\alpha_i ), \ \ \ \ \ z_2^{-1} z_1 z_2= z_1^{-1} z_2^{-1}z_1 z_2 z_1^{-1} z_2 z_1 \in G , \end{equation}
\begin{equation}\label{eq225} x_3 = (x_1 -x_2 )\cdot z_2 +x_2 , \ \ \ \ \ \ x_4 = (x_2 -x_1 )\cdot z_1 +x_1, \end{equation}
\begin{equation}\label{eq2215} (x_1 -x_2) \cdot (z_1 + z_2 - 1 )=(x_1 -x_2) \cdot (1-z_2^{-1} ) z_1 z_2 =(x_1 -x_2) \cdot (1-z_1^{-1} ) z_2 z_1 \in M.
\end{equation}
Accordingly, it follows from \eqref{eq225} that the set $ \mathrm{Col}_{X } (D_{f}) $ is generated by $x_1, x_2$.

Given a $G$-invariant trilinear form $\psi$,
it can be seen that the trilinear form $\mathcal{T}_{\psi} \bigl( ( x_1, x_2),( x_1', x_2' ),( x_1'', x_2'' ) \bigr)$ is expressed as
\[ \psi ((x_1-x_2) \cdot z_1 z_2^{-1}, \ x_2'-x_1', \ (x_1''-x_2'') \cdot (1-z_2^{-1}))+ \psi ((x_1-x_2) \cdot z_2^{-1}z_1 , \ (x_1'-x_2') \cdot (1-z_1), \ (x_1''-x_2'') \cdot (1-z_2^{-1})z_1). \]
\begin{rem}\label{eq3}
Here, we should give some examples from concrete $M$ and $\psi.$
In particular, the author attempted to get non-trivial trilinear form $\mathcal{T}_{\psi}$ when
$G$ is a Lie group and $M$ is a representation of $G$. However, even if $G=SL_2(\mathbb{R})$ or $G=SL_2(\mathbb{C})$ and $M= \mathbb{C}^2$ or $\mathbb{C}^3 $, the author computed the resulting $\mathcal{T}_{\psi}$ equal to zero.
Indeed, the author could not find non-trivial examples of $\mathcal{T}_{\psi} $ except those in \S \ref{secex3}.

Thus, it is a problem to find non-trivial examples of $\mathcal{T}_{\psi} $ from representations with respect to Lie groups.
\end{rem}

\subsection{The $(m,m)$-torus link $T_{m,m}$}\label{Lb12r32}
We also calculate the trilinear form $\mathcal{T}_{\psi}$ concerning the $(m,m)$-torus link,
following from Definition \ref{deals23}.
These calculations will be useful in the paper \cite{Nos4}, which
suggests invariants of ``Hurewitz equivalence classes".

Let $L$ be the $(m,m)$-torus link $T_{m,m}$ with $m \geq 2$,
and let $\alpha_1, \dots, \alpha_m$ be the arcs depicted in Figure \ref{ftf}.
Furthermore, let us identity $\alpha_{i+m}$ with $\alpha_i$ of period $m$.
By Wirtinger presentation, we have a presentation of $ \pi_L$ as
$$\langle \ a_1, \dots, a_m \ | \ a_1 \cdots a_m= a_m a_1 a_2 \cdots a_{m-1}= a_{m-1}a_m a_1 \cdots a_{m-2}= \cdots = a_2 \cdots a_m a_1 \ \rangle . $$

Given a homomorphism $f:\pi_L \ra G$ with $f(\alpha_i) \in G $,
let us discuss $X$-colorings $ \mathcal{C}$ over $f$.
Then, concerning the coloring condition on the $\ell $-th link component, it satisfies the equation
\begin{equation}\label{coleq} \bigl( \cdots (\mathcal{C}(\alpha_\ell ) \lhd \mathcal{C}(\alpha_{\ell+1})) \lhd \cdots \bigr)\lhd \mathcal{C}(\alpha_{\ell+m-1}) =\mathcal{C}( \alpha_\ell) ,\ \ \ \ \ \ \ {\rm for \ any \ }1 \leq \ell \leq m . \end{equation}
With notation $ \mathcal{C} (\alpha_i):= (x_i,z_i) \in X$, this equation \eqref{coleq} reduces to a system of linear equations
\begin{equation}\label{aac} ( x_{\ell-1} -x_{\ell} ) + \sum_{ \ell \leq j \leq \ell+m-2 }( x_j -x_{j+1} ) \cdot z_{j+1} z_{j+2} \cdots z_{m +\ell } = 0 \in M , \ \ \ \ \ \mathrm{for \ any \ } 1 \leq \ell \leq m .
\end{equation} 
Conversely, we can easily verify that, if a map $\mathcal{C}: \{ \mbox{arcs of $D$} \} \to X$ satisfies the equation \eqref{aac}, then $\mathcal{C} $ is an $ X$-coloring.
Denoting the left side in \eqref{aac} by $ \Gamma_{f,k}(\vec{x})$, consider a homomorphism
$$ \Gamma_{f}: M^m \lra M^m ; \ \ \ (x_1, \dots, x_m) \longmapsto (\Gamma_{f,1}(\vec{x}) , \dots, \Gamma_{f,m}(\vec{x}) ). $$
To conclude, the set $\mathrm{Col}_X(D_f )$ coincides with the kernel of $\Gamma_{f} $.

Next, we precisely formulate the resulting trilinear form. 
\begin{prop}\label{aa11c} 
Let $f : \pi_1(S^3 \setminus T_{m,m})\ra G$ be as above.
Let $\psi: M^3 \ra A$ be a $G$-invariant linear functions.
Then, the trilinear form $\mathcal{T }_{\psi}: \Ker (\Gamma_{f})^{\otimes 3} \ra A$ sends $(w_1, \dots, w_m) \otimes (x_1, \dots, x_m) \otimes ( y_1, \dots, y_m)$ to
\begin{equation}\label{bbbdd2}
\sum_{\ell=1}^{m } \sum_{k=1}^{m-1 } \psi \bigl( w_\ell \cdot (1 - z_\ell ) \hat{z}_{\ell+1; \ell+k -1 }, \ \sum_{j=1}^{k } (x_{j+\ell-1}-x_{j +\ell })\cdot \hat{z}_{j+\ell ;k+ \ell-1} ,\ y_{k +\ell } \cdot (1 - z_{k + \ell}^{-1}) \bigr) .
\end{equation}
Here, for $s \leq t$, we use notation $ \hat{z}_{s;t}:=z_s z_{s+1} \cdots z_t $ and $\hat{z}_{s+1; s}:=1 \in G .$
\end{prop}
\noindent
The formula is obtained by direct calculation and definitions.

\subsection{Examples of Theorem \ref{4i1}}\label{secex3}
We will give some examples in Theorem \ref{4i1}.
Thus, we should suppose the situation of Theorem \ref{4i1} as follows.
Let $G= \Z/3 = \langle t | t^3=1 \rangle $, and
$f : \pi_1(S^3 \setminus L) \ra \Z/3 $ be the map which sends every meridian to $t $.
Furthermore, take $M_i=A= \Z[t]/(n,t^2 +t +1) $ for some $n\in \Z_{\geq 0}$,
let $\psi_0 : M_1 \times M_2 \times M_3 \ra A $ send $(x,y,z)$ to $xyz.$

In this paragraph, we focus on only knots, $K$, such that $H^1(E_K,\partial E_K ; A ) \cong H^1(\widetilde{B}_K ;\Z/n) $ is isomorphic to either $A $ or $0$. We will write the trilinear map $\mathcal{T}_{\psi_0}$ as a cubic polynomial with respect to $a,b,c\in \bigl( H^1(E_K,\partial E_K ;\Z/n)\bigr)^3 $. 
Then, we give the resulting computation of $ \mathcal{T}_{\psi_0}$, when $K$ is a prime knot with crossing number $<7$. The list of the computation is as follows:

\begin{table}[h]\begin{center}
\begin{tabular}{ccc}
Knot & $n$ &$ \mathcal{T}_{\psi_0 }$ \\ \hline
$3_1$ & 2 & $abc $ \\ \hline
$4_1$ & 4 & $2a bc$ \\ \hline
$5_1$ & 3 & $0$ \\ \hline
$5_2$ & 5 & $(1+t)a bc$ \\ \hline
$6_1$ & any & $ 0$ \\ \hline
$ 6_2$ & any & $0 $ \\ \hline
$ 6_3$ & any & $0 $ \\ \hline
\end{tabular}\end{center}
\end{table}

\section{Proofs of the theorems}\label{Lbo33222}
We will complete the proofs of
Theorems \ref{mainthm3}--\ref{mainthm4} in \S \ref{pr1}. 
While the statements were described in terms of ordinary cohomology,
the proof will be done via the group cohomology.
For this purpose, in \S \ref{KI11}, we review the relative group homology.

\subsection{Preliminary; Review of relative group cohomology}\label{KI11}
We will spell out the relative group (co)homology in the non-homogeneous terms. 
Throughout this subsection, we fix a group $ \Gamma $ and a homomorphism $f : \Gamma \ra G$.
Then, we have the action of $\Gamma$ on the right $G$-module $M$ via $f$.

Let $ C^{n}_{\rm gr }(\Gamma;M ) $ be $ \mathrm{Map} ( \Gamma^n , M ) $.
For $ \phi \in C^{n}_{\rm gr }(\Gamma;M )$, define the coboundary $\partial^n( \phi) \in C_{\mathrm{gr}}^{n+1} (\Gamma;M)$ by the formula $\partial^n( \phi) (g_1,\dots, g_{n+1})= $
\[
\phi ( g_2, \dots ,g_{n+1}) +\!\sum_{1 \leq i \leq n}\!\! (-1)^i \phi ( g_1, \dots ,g_{i-1}, g_{i} g_{i+1}, g_{i+2},\dots , g_{n+1})+(-1)^{n} \phi ( g_1, \dots , g_{n}) g_{n+1} .\]
Furthermore, we set subgroups $ K_j $ and the inclusions $\iota_j: K_j \hookrightarrow \Gamma $,
where the index $ j$ runs over $ 1 \leq j \leq m $. 
Then, we can define the mapping cone of $ \iota_j$:
More precisely, 
$$C^n( \Gamma, K_\mathcal{J}; M ):= \mathrm{Map} ( \Gamma^n , M ) \oplus \bigl( \bigoplus_{j } \mathrm{Map }((K_{j})^{n-1} , M ) \bigr) . $$
For $(h,k_1, \dots, k_m ) \in C^n( \Gamma, K_\mathcal{J}; M ) $, let us define
$ \partial^n(h,k_1, \dots, k_m )$ in $ C^{n+1}( \Gamma, K_\mathcal{J}; M )$ by 
$$ \partial^n \bigl(h,k_1, \dots, k_m \bigr)( a, b_1, \dots, b_m)= \bigl( \partial^{n} h( a), \ h (b_1) -\partial^{n-1} k_1(b_1), \dots,h (b_m) -\partial^{n-1} k_m(b_m)\bigr) ,$$
where $( a, b_1, \dots, b_m) \in \Gamma^{n+1} \times K_1^{n} \times \cdots \times K_m^{n} $. Then we have a complex $ (C^*( \Gamma, K_\mathcal{J}; M ), \partial^*)$, and
can define the cohomology.

We now observe
the submodule consisting of 1-cocycles $Z^1( \Gamma, K_\mathcal{J}; M ) $.
Let us define the semi-direct product $M \rtimes G $ by
$$ (a, g) \star (a',g'):=( a \cdot g' + a', \ gg'), \ \ \ \ \mathrm{for} \ \ a,a' \in M, \ \ \ g,g'\in G. $$
Let $ \Hom_f (\Gamma , M \rtimes G )$ be the set of group homomorphisms $\Gamma \ra M \rtimes G $ over the homomorphism $f$.
Consider a map
$$ Z^1( \Gamma, K_\mathcal{J}; M ) \ra \Hom_f (\Gamma , M \rtimes G ) \oplus M^m; \ \ \ \ \ \ (h,y_1,\dots, y_m) \mapsto (\gamma \mapsto ( h(\gamma), f(\gamma)),y_1,\dots, y_m) .$$
\begin{lem}[{\cite[Lemma 5.2]{Nos5}}]\label{clAl1}
This map gives an isomorphism 
between $Z^1( \Gamma, K_\mathcal{J}; M )$ and the following set:
$$ \bigl\{ \ (\widetilde{f} , y_1, \dots, y_m ) \in \Hom_f ( \Gamma, M \rtimes G) \oplus M^m \ \bigl| \ \ \widetilde{f} (h_j) = ( y_j - y_j \cdot h_j, \ f_j( h_j) ) , \ \ \mathrm{for \ any } \ h_j \in K_j \ \bigr\} . $$

Moreover, the image of $\partial^1$, i.e., $B^1( \Gamma, K_\mathcal{J}; M )$, is equal
to the subset $\{ ( \widetilde{f}_a, a, \dots, a)\}_{a \in M}. $
Here, for $a \in M,$ the map
$ \widetilde{f}_a : \Gamma \ra M \rtimes G $ is defined as a homomorphism which sends $\gamma $ to $( a - a \cdot \gamma, \ f (\gamma ))$.
In particular, if $ K_\mathcal{J} $ is non-empty, $B^1( \Gamma, K_\mathcal{J}; M )$ is a direct summand of $ Z^1( \Gamma, K_\mathcal{J}; M )$.
\end{lem}

Furthermore, we review the cup product.
When $ K_\mathcal{J} $ is the empty set, the product of $u \in C^p( \Gamma; M )$ and $v \in C^{q}( \Gamma; M' )$ is defined to be 
$u \smile v \in C^{p+q} ( \Gamma; M \otimes M')$ given by
\begin{equation}\label{cupprod} ( u \smile v) ( g_1 ,\dots, g_{p+q}):= (-1)^{pq} \bigl( u (g_1 ,\dots, g_{p} ) g_{p+1} \cdots g_{p+q} \bigr) \otimes v (g_{p+1} ,\dots, g_{p+q} ) . \end{equation}
Furthermore, if $ K_\mathcal{J} $ is not empty, for two elements $(f,k_1, \dots, k_m )\in C^p( \Gamma, K_\mathcal{J} ; M )$ and
$(f',k'_1, \dots, k'_m )\in C^{q}( \Gamma, K_\mathcal{J} ; M' )$, let us define {\it the cup product} to be the formula
$$ ( f \smile f', \ k_1\smile f', \dots, \ k_m\smile f') \in C^{p+q}( \Gamma, K_\mathcal{J} ; M \otimes M'). $$
This formula yields a bilinear map, by passage to cohomology.

Finally, we observe another complex. Consider the module of the form
$$ C^n_{\rm red}(\Gamma ):= \{ \ (c_1, \dots, c_m ) \in \mathrm{Map}( \Z[\Gamma ^n], M)^m \ | \ c_1 +c_2 + \cdots +c_m =0\in \mathrm{Map}( \Z[\Gamma ^n], M)\ \} .$$
Then, this complex canonically has an inclusion into the direct sum of $ C^n(\Gamma, K_j)$:
$$P_n: C^n_{\rm red}(\Gamma ) \lra \bigoplus_{j : \ 1 \leq j \leq m} C^n(\Gamma, K_j). $$
Then, we define a quotient complex, $D^n( \Gamma, K_\mathcal{J}; M )$, to be
the cokernel of $ P_n$.
Then, $C^n( \Gamma, K_\mathcal{J}; M )$ is isomorphic to $D^n( \Gamma, K_\mathcal{J}; M ) $, since the kernel of the inclusions $\oplus_{j=1}^{ m} C^n(\Gamma, K_j) \ra C^n(\Gamma, \mathcal{K}) $ is the image of $ P_n $.
\begin{rem}\label{clAl221}
We give a natural relation to usual cohomology. 
Take the Eilenberg-MacLane spaces of $ \Gamma$ and of $K_j$, and
consider the map $ (\iota_j)_* : K(K_j,1 ) \ra K( \Gamma,1 ) $ induced by the inclusions.
Then the relative homology $H_n( \Gamma, K_\mathcal{J} ; M ) $ is isomorphic to the homology of the mapping cone of $ \sqcup_j K(K_j,1 )\ra K(\Gamma ,1 )$
with local coefficients.
Further, the cup product $\smile$ above coincides with that on
the singular cohomology groups. 

We mention the case where $L$ is either a knot or a hyperbolic link.
Then, the complement $S^3 \setminus L$ is known to be an Eilenberg-MacLane space.
Since we only use $\Gamma$ as $\pi_1(S^3 \setminus L)$ in this paper,
we may discuss only the relative group cohomology.
\end{rem}


\subsection{Review; results of the previous papers \cite{Nos5} and \cite{Nos6}.}\label{pr2}
Throughout this section, we denote $\pi_1(S^3 \setminus L)$ by
$\pi_L$, and the union of the fundamental groups of the boundaries of $S^3 \setminus L$ by $\partial \pi_L$, for brevity. Let $m= \# L $, and choose a diagram $D.$

\begin{thm}[{\cite[Theorem 2.2]{Nos5}}]\label{290}
Let $X$ be $M \times G$, as mentioned in \eqref{kihon}.
Let $\kappa : X \ra M \rtimes G $ be a map which sends $(m,g)$ to $ (m-mg , g)$.
Given an $X$-coloring $\mathcal{C}$ over $f $,
consider a map $\{\mathrm{arcs \ of \ }D \} \ra M \rtimes G$
which assigns $ \alpha $ to $ \kappa \bigl( \mathcal{C}(\alpha) \bigr) $.
This assignment yields isomorphisms
$$ \mathrm{Col}_X(D_{f}) \cong Z^1( \pi_L ,\partial \pi_L ; M ), \ \ \ \ \mathrm{Col}^{\rm red}_X(D_{f}) \cong H^1( \pi_L ,\partial \pi_L ; M ). $$
\end{thm}

Next, we explain Theorem \ref{oo}.
Choose a relative 1-cocycle $\tilde{f}: \pi_L \ra M \rtimes \pi_L$
with $ y_1,\dots, y_{m}$.
We define the subgroup $K_{\ell} $ to be
$$ \{ (y_{\ell}-y_{\ell} \mathfrak{m}_{\ell}^a \mathfrak{l}_{\ell}^b , \mathfrak{m}_{\ell}^a \mathfrak{l}_{\ell}^b ) \in M \rtimes \pi_L \ | \ a,b \in \Z^2 \ \}. $$
Furthermore, given a $G$-invariant trilinear map $\psi: M^3 \ra A$,
consider the map
\[\theta_{\ell}: (M \rtimes \pi_1(S^3 \setminus L) )^3 \lra A ; \]
\begin{equation}\label{aaddd} ((a,g),(b,h),(c,k)) \longmapsto \psi ( (a +y_{\ell} -y_{\ell} g)\cdot hk , (b+y_{\ell} -y_{\ell} h) \cdot k, c+y_{\ell} -y_{\ell} k) .\end{equation}
Then, we can easily check that each $\theta_{\ell} $ is a 3-cocycle in $C^3( M \rtimes \pi_L;A)$.
Then, the collection $\Psi:= (\theta_{1},\dots, \theta_{\# L} ) $ represents
a relative 3-cocycle in $D^3( M \rtimes \pi_L,\mathcal{K} ;A) $.
\begin{prop}[{\cite[Proposition 6.7]{Nos6}}]\label{oo}
Under the notation above,
fix a shadow coloring $ \mathcal{S}_{\tilde{f}} $ corresponding the relative 1-cocycle $(\tilde{f},y_1,\dots,y_{\# L}).$

If $L$ is either a hyperbolic link or a prime knot which is neither a cable knot nor a torus knot, as in Theorem \ref{mainthm3}, then 
the diagonal restriction of $ \mathcal{T}_{ \psi}$ is equal to
the pairing of the 3-class $[E_L,\partial E_L]$ and the above 3-cocycle $\Psi$. To be precise,
\begin{equation}\label{aref} \mathcal{T}_{ \psi}( \mathcal{S}_{\tilde{f}}, \mathcal{S}_{\tilde{f}}, \mathcal{S}_{\tilde{f}})=\psi \langle \Psi , \tilde{f}_*[E_L,\partial E_L] \rangle . \end{equation} 
Furthermore, if $L$ is the $(m,n)$-torus knot, the same equality \eqref{aref} holds modulo $mn.$
\end{prop}
\subsection{Proof of Theorem \ref{mainthm3}; trilinear pairing }\label{pr1}


\begin{proof}[Proof of Theorem \ref{mainthm3}.]
First, we observe \eqref{are66f} below.
Consider a 0-cochain $\vec{y}:= (y_1,\dots, y_{\# L}) \in D^0(M \rtimes \pi_L, M)$.
Then, $ \widetilde{f} -\partial^{0} \vec{y}$ is represented by another 1-cocycle
$$ \mathcal{C}' :=((\widetilde{f} -\bar{y}_1 , \dots, \widetilde{f} -\bar{y}_{\# L} ),(0,\dots, 0) ) \in D^1(M \rtimes \pi_L, M),$$
where $\bar{y}_{\ell}$ means a map $\pi_L \ra M $ which
takes $g$ to $ y_1 -y_1 g $.
Then, the 3-cocycle $ \Psi$ explained in \eqref{aaddd}
is equal to the cup product $ \mathcal{C}' \smile \mathcal{C}' \smile \mathcal{C}' $, by definition.
Hence, Proposition \ref{oo} implies
\begin{equation}\label{are66f} \mathcal{T}_{ \psi}( \mathcal{S}_{\tilde{f}}, \mathcal{S}_{\tilde{f}}, \mathcal{S}_{\tilde{f}})=\psi \langle \mathcal{C}' \smile \mathcal{C}'\smile \mathcal{C}' , [E_L,\partial E_L] \rangle =\psi \langle \mathcal{C} \smile \mathcal{C}\smile \mathcal{C} , [E_L,\partial E_L] \rangle . \end{equation}


Finally, we will deal with non-diagonal parts, and complete the proof.
Here, we define $M$ to be the direct product $M_1 \times M_2 \times M_3 $,
and consider the $j$-th inclusion
$$ \iota_j: M_j \lra M= M_1 \times M_2 \times M_3; \ \ \ \ x \longmapsto (\delta_{1j} x,\delta_{2j} x,\delta_{3j} x ). $$
Thus, we can decompose $\mathcal{S}_{\tilde{f}}$ as $( \mathcal{S}_{1}, \mathcal{S}_{2}, \mathcal{S}_{3}) \in \mathrm{Col}_{X_1} (D_f )\times \mathrm{Col}_{X_2} (D_f )\times \mathrm{Col}_{X_3} (D_f )$ componentwise.
In addition, we define a $G$-invariant trilinear form
$$ \overline{\psi} : M \times M \times M \lra A; \ \ \ \ \bigl( (a,b,c),(d,e,f),(g,h,i)\bigr) \longmapsto \psi (a,e,f) . $$

Then, the transformation of the coefficients $\iota_1 \times \iota_2 \times \iota_3$ yields a diagram
$${\normalsize
\xymatrix{
\prod_{i=1}^3 H^1( E_L, \partial E_L ; M_i) \ar[r]^{\!\!\!\!\! \smile}\ar[d]_{ }^{(\iota_1 \times \iota_2 \times \iota_3)_*} & H^3( E_L, \partial E_L ; M_1 \times M_2 \times M_3 ) \ar[rrr]^{\ \ \ \ \ \ \ \ \ \ \ \ \psi \circ \langle \bullet, [E_L,\partial E_L] \rangle }
\ar[d]^{(\iota_1 \times \iota_2 \times \iota_3)_*} & & & A \ar@{=}[d] \\
H^1( E_L ,\partial E_L; M) \ar[r]^{\!\!\!\!\!\!\!\! \smile_{\Delta }} & H^3 ( E_L ,\partial E_L; M \times M \times M) \ar[rrr]^{\ \ \ \ \ \ \ \ \ \ \ \ \overline{\psi} \circ \langle \bullet, [E_L,\partial E_L] \rangle }& & & A .
}}
$$
Here, the left bottom $\smile_{\Delta } $ is defined by $a \mapsto a \smile a \smile a$.
Then, we can verify the commutativity by definitions.
By Proposition \ref{oo}, the bottom arrow is equal to the left hand side in \eqref{aref}. Hence, the pullback to $\prod_{i=1}^3 H^1( E_L, \partial E_L ; M_i) $
is equal to the trilinear $\mathcal{T}_{\psi}$ as desired.
\end{proof}
\begin{proof}[Proof of Theorem \ref{mainthm4}.]
Let $L$ be the $(m,n)$-torus knot.
According to the latter part in Theorem \ref{oo},
we need discussions modulo $mn$. However, the proof runs well in the same manner.
\end{proof}
\subsection*{Acknowledgments}
The work is partially supported by JSPS KAKENHI Grant Number 17K05257.

\vskip 1pc

\normalsize
DEPARTMENT OF
MATHEMATICS
TOKYO
INSTITUTE OF
TECHNOLOGY
2-12-1
OOKAYAMA
, MEGURO-KU TOKYO
152-8551 JAPAN

\normalsize
E-mail: nosaka@math.titech.ac.jp

\end{document}